\documentclass[11pt]{amsart}
\usepackage{amsmath,amssymb,amsthm}
\usepackage[margin=1.15in]{geometry}
\usepackage{enumitem}
\usepackage{xcolor}
\usepackage{tikz}
\usepackage[colorlinks=true,linkcolor=blue,citecolor=blue,urlcolor=blue]{hyperref}
\usepackage[color=yellow!30]{todonotes}

\usepackage{mathrsfs}

\newtheorem{theorem}{Theorem}[section]
\newtheorem{lemma}[theorem]{Lemma}

\newtheorem{corollary}[theorem]{Corollary}
\newtheorem{claim}[theorem]{Claim}
\theoremstyle{definition}
\newtheorem{definition}[theorem]{Definition}
\newtheorem{observation}[theorem]{Observation}
\newtheorem{remark}[theorem]{Remark}

\newtheorem{question}[theorem]{Question}

\newcommand{\Fin}{\mathrm{fin}}
\newcommand{\fin}{\mathrm{fin}}
\newcommand{\Pw}{\mathcal{P}}

\newcommand{\hgt}{\operatorname{ht}}
\newcommand{\supp}{\operatorname{supp}}
\newcommand{\otp}{\operatorname{otp}}
\newcommand{\cf}{\operatorname{cf}}
\newcommand{\eqs}{=^{*}}
\newcommand{\neqs}{\neq^{*}}
\newcommand{\force}{\Vdash}
\newcommand{\rest}{\restriction}
\newcommand{\nm}[1]{\dot{#1}}          
\newcommand{\FG}{F_{G}}

\title[A non-trivially non-trivial automorphism]{A non-trivially non-trivial automorphism at an inaccessible cardinal}
\author{M\'{a}rk Po\'{o}r}
\address{Department of Mathematics, Cornell University, 310 Malott Hall, Ithaca, NY 14853-4201, USA}
\address{HUN-REN Alfr\'ed R\'enyi Institute of Mathematics, Re\'altanoda utca 13--15, H-1053 Budapest, Hungary}
\email{sokmark@renyi.hu}
\author{Jing Zhang}
\address{Department of Mathematics, University of North Texas, 1155 Union Circle \#311430, Denton, TX 76203-5017, USA}
\email{jing.zhang2@unt.edu}
\thanks{The first author was supported by the National Research, Development and Innovation Office -- NKFIH, grants no. ~124749, ~152822.}
\date{}

\begin{document}

\begin{abstract}
We show that, relative to the existence of an inaccessible cardinal, it is consistent that for an inaccessible cardinal $\kappa$, there is a non-trivial automorphism of $\Pw(\kappa)/\fin$ whose restriction to $\Pw(\beta)/\Fin$ is induced by a bijection on $\beta$ for every $\beta<\kappa$. In particular, the conjunction of Todor\v{c}evi\'{c}'s Open Coloring
Axiom together with Martin's Axiom is compatible with the existence of a non-trivial automorphism on $\Pw(\kappa)/\Fin$ where $\kappa$ is the least inaccessible cardinal.
This answers questions posed by Larson--McKenney, Shelah--Stepr\=ans, and Veli\v{c}kovi\'{c}.
\end{abstract}

\maketitle

\section{Introduction}\label{sec:intro}

A basic rigidity question for a quotient Boolean algebra is whether its
automorphisms are induced by maps on the underlying set. Let $X$ be an
infinite set and let $I$ be a proper ideal on $X$. The quotient
$\Pw(X)/I$ identifies two sets $A,B\subseteq X$ when
$A\mathbin{\triangle}B\in I$, with Boolean operations induced by the
usual set operations: union ($\vee$), intersection ($\wedge$) and complement ($\neg$). An automorphism $\Phi$ of $\Pw(X)/I$ is
\emph{trivial} if there is a function $g\colon X\to X$ such that
\[
\Phi([A]_I)=[g[A]]_I
\qquad\text{for every }A\subseteq X.
\]
We are primarily concerned with the ideal 
$\mathbf{fin}$ of finite sets.
For an infinite cardinal $\kappa$ and an infinite cardinal
$\lambda\leq\kappa$, we also write $I_\lambda=[\kappa]^{<\lambda}$.

In the 1950s, Walter Rudin \cite{Rudin1,Rudin2} initiated the study of
automorphisms of $\Pw(\omega)/\Fin$ and proved that the Continuum Hypothesis
implies the existence of a non-trivial automorphism of $\Pw(\omega)/\fin$. His
motivation was to construct a non-trivial autohomeomorphism of
$\beta\mathbb N\setminus\mathbb N$, with Stone duality giving the equivalent
formulation in terms of its clopen algebra, $\Pw(\omega)/\Fin$. A major breakthrough was made by
Shelah \cite{properforcing}, who proved that it is consistent with ZFC that
every automorphism of $\Pw(\omega)/\Fin$ is trivial. Later, Shelah and Stepr\={a}ns \cite{pfa} derived the same conclusion from the Proper Forcing Axiom (PFA). Veli\v{c}kovi\'{c}
\cite{Velickovic} weakened the hypothesis further to the conjunction of Todor\v{c}evi\'{c}'s Open Coloring
Axiom ($\mathrm{OCA}_{\mathrm T}$) together with Martin's Axiom
($\mathrm{MA}_{\aleph_1}$). A recent improvement by De Bondt, Farah, and
Vignati \cite{debondt2024trivialisomorphismsreducedproducts} removes the
assumption $\mathrm{MA}_{\aleph_1}$ and derives the conclusion from
$\mathrm{OCA}_{\mathrm T}$ alone. The study of $\Pw(\omega)/I$ for other
ideals $I$ on $\omega$ is closely connected with descriptive set theory and
has developed into a rich subject; see \cite{FarahBook} and \cite{FarahBookUpdated} for an update. A parallel and very
successful line of research concerns automorphisms of the Calkin algebra
$\mathcal B(H)/\mathcal K(H)$, where $H$ is a separable infinite-dimensional
Hilbert space and $\mathcal B(H)$ and $\mathcal K(H)$ denote, respectively,
the bounded and compact operators on $H$. Results of Phillips--Weaver \cite{PhillipsWeaver} and Farah \cite{FarahInner} show that the existence of an outer automorphism of the Calkin
algebra is independent of ZFC (a well-known question asked in \cite{CalkinQuestion}). We refer the reader to \cite{FarahBookNew} for further
details.

The main objects studied in this paper are automorphisms of
$\Pw(\kappa)/\Fin$, where $\kappa$ is an infinite cardinal. Veli\v{c}kovi\'{c}
\cite{Velickovic} showed that $\mathrm{OCA}_{\mathrm T}$ and
$\mathrm{MA}_{\aleph_1}$ imply that all automorphisms of
$\Pw(\omega)/\Fin$ and $\Pw(\omega_1)/\Fin$ are trivial. He also showed that
the Proper Forcing Axiom ($\mathrm{PFA}$) implies the corresponding statement
for every $\kappa\geq\omega$. It is therefore natural to ask whether
$\mathrm{OCA}_{\mathrm T}+\mathrm{MA}_{\aleph_1}$ suffices for the conclusion for all
$\kappa\geq \omega$. In the final paragraph of \cite[p.~13]{Velickovic}, Veli\v{c}kovi\'{c}
conjectured that a non-trivial automorphism of
$\Pw(\omega_2)/\Fin$ can be constructed from $\mathrm{MA}_{\aleph_1}$
together with $\diamondsuit^*(\omega_2)$ and $\square_{\omega_1}$, the latter
two of which hold in G\"odel's constructible universe. However, later work
of Shelah and Stepr\={a}ns \cite{ShSt16} and Larson and McKenney \cite{LM16}
shows in ZFC that, for every uncountable $\kappa$ below the first inaccessible cardinal,
an automorphism $\Phi\colon\Pw(\kappa)/\Fin\to\Pw(\kappa)/\Fin$ is trivial
provided it is trivial on $\Pw(X)/\Fin$ for every
$X\in[\kappa]^{\aleph_1}$. Here $\Phi$ is \emph{trivial on}
$\Pw(X)/\Fin$ if its restriction to $\Pw(X)/\fin$ is induced by a function
from $X$ to $\kappa$. Since the conjunction $\mathrm{OCA}_{\mathrm T}+\mathrm{MA}_{\aleph_1}$ is compatible with $\diamondsuit^*(\omega_2)$ and $\square_{\omega_1}$, this refutes the above conjecture. The two
distinguished cardinals in these results are $\omega_1$ and the first
inaccessible cardinal.

This prompted Larson and McKenney to ask the following:
\begin{question}[{\cite[Question 1 (a)]{LM16}}]\label{question: LM}
Is it consistent that there exist an uncountable $\kappa$ and an automorphism on $\Pw(\kappa)/\fin$ which is not trivial on any cocountable set?
\end{question}

\begin{question}[{\cite[Question 1 (g)]{LM16}}]\label{question: LM2}
Is it consistent that there are uncountable cardinals $\lambda\leq \kappa$ and a nontrivial automorphism on $\Pw(\kappa)/I_\lambda$?
\end{question}

For $\lambda=\kappa$, Kellner and Shelah \cite{KSh26} proved that
if $\kappa$ is inaccessible and $2^\kappa=\kappa^+$, then
$\Pw(\kappa)/I_\kappa$ has a nowhere trivial automorphism.
Thus if we have no restriction on $\lambda$, Question \ref{question: LM2} already has a positive
answer. Our construction gives a single model that contains nontrivial automorphisms on
$\Pw(\kappa)/I_\lambda$, for every infinite $\lambda\leq\kappa$, including the ideal
of countable sets.

Following Shelah and Stepr\={a}ns \cite{ShSt16}, we use the
following terminology.

\begin{definition}
Let $\kappa$ be an uncountable cardinal. An automorphism $\Phi$ of
$\Pw(\kappa)/\Fin$ is \emph{non-trivially non-trivial} if it is
nontrivial, but its restriction to $\Pw(X)/\Fin$ is trivial for
every $X\in[\kappa]^{<\kappa}$.
\end{definition}

Shelah and Stepr\={a}ns \cite{ShStomega1} proved the consistency
of a non-trivially non-trivial automorphism of
$\Pw(\omega_1)/\Fin$. Their result already gives a positive answer
to Question~\ref{question: LM}. We address the corresponding problem, which they left open, at an
inaccessible cardinal.

\begin{question}[{\cite[Question 7.4]{ShSt16}}, {\cite[Question 5.2]{ShStomega1}}]\label{question: SS}
Is it consistent, relative to the consistency of an inaccessible cardinal, that there is a non-trivially non-trivial automorphism of
$\Pw(\kappa)/\fin$ where $\kappa$ is inaccessible?
\end{question}

\begin{question}[Veli\v{c}kovi\'{c}, {\cite[p.~167]{ShSt16}}]\label{question: V}
Do $\mathrm{MA}_{\aleph_1}$ and $\mathrm{OCA}_{\mathrm{T}}$ imply that,
for every cardinal $\kappa\geq\omega$, every automorphism of
$\Pw(\kappa)/\fin$ is trivial?
\end{question}

These questions are especially interesting in light of the following theorem: 

\begin{theorem}[see {\cite[Corollary 1.2]{ShSt16}}]
$\mathrm{MA}_{\aleph_1}$ and $\mathrm{OCA}_{\mathrm{T}}$ imply that, if $\kappa$ is less than the first inaccessible cardinal, then every automorphism of $\Pw(\kappa)/\fin$ is trivial.
\end{theorem}

The credit for the theorem is as follows:
	\begin{enumerate}
		\item Veli\v{c}kovi\'{c} \cite{Velickovic} showed that $\mathrm{MA}_{\aleph_1}$ and $\mathrm{OCA}_{\mathrm{T}}$ imply the statement when $\kappa=\omega$ and $\kappa=\omega_1$;
	\item Shelah and Stepr\={a}ns \cite{ShSt16} showed that for $\kappa>2^{\aleph_0}$ below the first inaccessible cardinal, any automorphism $\Phi$ on $\Pw(\kappa)/\fin$ is trivial outside a set of size $2^{\aleph_0}$, i.e.\  $\Phi$ is trivial on $\Pw(X)/\fin$ for some appropriate  $X \subseteq 
	\kappa$ satisfying $|\kappa-X|\leq 2^{\aleph_0}$;
		\item Larson and McKenney \cite{LM16} showed that any automorphism $\Phi$ on $\Pw(\kappa)/\fin$ is trivial given that $\kappa\leq 2^{\aleph_0}$ and $\Phi\restriction \Pw(X)/\fin$ is trivial for any $X\in [\kappa]^{\aleph_1}$.
	\end{enumerate}

We answer all these questions in this paper. In particular, we answer Questions 
\ref{question: LM}, \ref{question: LM2}, and \ref{question: SS} positively and Question \ref{question: V} negatively. The main result of this paper is: 

\begin{theorem}\label{theorem: main}
Relative to the existence of an inaccessible cardinal $\kappa$, it is consistent that $\kappa$ is inaccessible and there is a non-trivially non-trivial automorphism $\Phi$ on $\Pw(\kappa)/\fin$. Moreover, $\Phi$ induces an automorphism $\Phi_\lambda$ of
$\Pw(\kappa)/I_\lambda$ for every infinite cardinal
$\lambda\leq\kappa$, and
$\Phi_\lambda\restriction\Pw(X)/I_\lambda$ is nontrivial for
any $X\in[\kappa]^\kappa$.
\end{theorem}

The forcing used in Theorem~\ref{theorem: main} also yields the
following result.

\begin{corollary}\label{cor:oca-ma}
Relative to the existence of an inaccessible cardinal, it is consistent that $\mathrm{OCA}_{\mathrm T}+\mathrm{MA}_{\aleph_1}$
holds and there exists a non-trivial automorphism on $\Pw(\kappa)/\fin$ where $\kappa$ is the least inaccessible cardinal.
\end{corollary}

Our construction is inspired by that of Shelah and Stepr\={a}ns \cite{ShStomega1}, but it requires substantial modification. The framework in \cite{ShStomega1} relies on Jensen's $\diamondsuit^+$ principle: the final model is obtained by forcing with a suitable Suslin tree built using $\diamondsuit^+$. Extending this approach
to an inaccessible cardinal requires control of both the tree
and the local maps that determine the automorphism. Several obstacles appear and linger in this approach. 

We overcome these difficulties by avoiding the $\diamondsuit$-style construction altogether. Instead, we use a carefully designed forcing iteration to add the desired non-trivially non-trivial automorphism. The construction starts with adding a coherent family of local involutions. The main
difficulty is to ensure that these local maps determine a non-trivial
automorphism of the entire quotient. To achieve this, we follow with a forcing iteration. The main point is to show this coherent family of local involutions remains non-trivial after the iteration. A key new ingredient of our technique is a nonstandard way of taking inverse limits in forcing iterations. We expect this method to apply to a wide range of problems, and we think it deserves further exploration.

%

The organization of the paper is as follows: 

\begin{enumerate}
\item Section \ref{section: preliminary} contains some preliminary facts. 
\item Section \ref{sec:iteration} defines the main forcing recursively used in Theorem \ref{theorem: main}.
\item Section \ref{sec:mainlemma} gives the proof of the main properties of the forcing and finishes the proof of Theorem \ref{theorem: main}. A discussion on how our main result answers Questions \ref{question: LM}, \ref{question: LM2}, \ref{question: SS} and \ref{question: V} is included.

\item Section \ref{section: questions} concludes with some questions and remarks.
\end{enumerate}

\section{Preliminaries, notations}\label{section: preliminary}

Throughout the paper we say the sets $A$ and $B$ are almost equal, in symbols,  $A \eqs B$ if they are equal modulo a finite error, that is, $A \Delta B = (A \setminus B) \cup (B \setminus A)$ is finite. This notation is meaningful when applied to functions $g,h$, and $g \eqs h$ holds exactly when $\{x \in \textrm{dom}(g) \cup \textrm{dom}(h): \ g(x)  \neq h(x)\}$ is finite. To enhance readability, we may as well write $A - B$ for $A \setminus B$. If $f$ is a function and $A$ is a subset of the domain of $f$, we write $f[A]$ for the pointwise image of $A$ under $f$, $\{f(a): \ a \in A\}$.

We isolate some standard facts used in later sections.

\begin{lemma}\label{lem:agreement}
Let $B$ be infinite and $g,h\colon B\to\kappa$ injections with
$g\neqs h$. Then for some infinite $H \subseteq B$ $g[H] \cap h[H] = \emptyset$. In particular, if for a pair of injections $g$ and $h$, $g[H] \eqs h[H]$ holds for every $H \subseteq B$, then $g \eqs h$.
\end{lemma}

\begin{proof}
Since $D:=\{x\in B:g(x)\ne h(x)\}$ is infinite we can fix a countably infinite
$A\subseteq D$ and form the graph on $A$ joining $x\ne y$ iff $g(x)=h(y)$ or
$g(y)=h(x)$. Every vertex has degree $\le2$, so $A$ has an infinite
independent set $H$; then $g[H]\cap h[H]=\emptyset$.
\end{proof}

\begin{lemma}\label{lem:localimage}
Let $\langle f_\delta:\delta<\kappa\rangle$ be involutions
$f_\delta\colon\delta\to\delta$ with $f_{\delta}\rest\delta'\eqs f_{\delta'}$
for $\delta'<\delta$. If $\delta'\le\delta$ and $x,y\subseteq\kappa$ satisfy
$y\cap\delta\eqs f_{\delta}[x\cap\delta]$, then
$y\cap\delta'\eqs f_{\delta'}[x\cap\delta']$.
\end{lemma}

\begin{proof}
$y\cap\delta'=(y\cap\delta)\cap\delta'
\eqs f_{\delta}[x\cap\delta]\cap\delta'$. Now
$f_{\delta}[x\cap\delta]\cap\delta'$ differs from
$f_{\delta'}[x\cap\delta']$ by a finite set: the points of $x\cap\delta'$
where $f_\delta$ and $f_{\delta'}$ differ are finitely many; and the points
of $x\cap[\delta',\delta)$ mapped below $\delta'$ by $f_\delta$ are finitely
many, since $f_\delta$ is an involution and
$f_\delta\rest\delta'\eqs f_{\delta'}$.
\end{proof}

\begin{lemma}\label{lemma: separate}
Let $g$ and $h$ be functions sharing the same domain $H$ such that $g\ne^*h$. Then there exists an infinite subset $B\subseteq H$ such that $g[B] \cap h[B] =\emptyset$. 
\end{lemma}

\begin{proof}
Let $D=\{x\in H:  g(x)\neq h(x)\}$. By the hypothesis, $D$ is infinite.
If some value $v$ has $g^{-1}(v)\cap D$ infinite, take
$B:=g^{-1}(v)\cap D$: then $g[B]=\{v\}$ while $v\notin h[B]$, since
$h(x)\ne g(x)=v$ for every $x\in B$. Symmetrically if some $h$-fiber
meets $D$ infinitely, the same argument applies. Otherwise $g$ and $h$ are finite-to-one on $D$, both have infinite ranges, and by passing to a smaller infinite set we may assume that $g$, $h$ are injective, and Lemma \ref{lem:agreement} applies (in fact the argument in Lemma \ref{lem:agreement} itself works for finite-to-one functions).
\end{proof}

We refer the reader to the standard textbooks \cite{Kunen,Jech} for notation and basic facts concerning forcing.

\section{The main forcing}\label{sec:iteration}
In this section, we work towards the proof of Theorem \ref{theorem: main} by defining our forcing.
Assume from now on: $\kappa$ is inaccessible and $2^{\kappa}=\kappa^{+}$. We will define recursively a sequence $\langle\mathbb P_\alpha:\alpha\le
\kappa^{+}\rangle$ together with, for each $\gamma\ge 1$, a nice
$\mathbb P_\gamma$-name $\nm X_\gamma$ for a subset of $\kappa$, so that the
following hold.

\begin{enumerate}[label=\textup{(P\arabic*)},start=0]
\item $\mathbb P_\alpha\lessdot\mathbb P_\beta$ whenever
$\alpha<\beta\le\kappa^{+}$; namely, $\mathbb P_\alpha$ is a complete suborder of $\mathbb P_\beta$.

\item $\mathbb P_\alpha$ consists of $p=\langle p_\gamma:\gamma<\alpha\rangle$
where:
\begin{itemize}
\item $p_0=\langle f^{p}_\delta:\delta\le\Theta\rangle$ for some
$\Theta<\kappa$, denoted $\hgt(p)$; each $f^{p}_\delta\colon\delta\to\delta$
is an \emph{involution}, namely, $f^{p}_\delta\circ f^{p}_\delta =\mathrm{id}_{\delta}$; and $f^{p}_\delta\rest\delta'\eqs f^{p}_{\delta'}$
for $\delta'<\delta\le\Theta$;
\item for $1\le\gamma<\alpha$, $p\restriction \gamma\in \mathbb{P}_\gamma$ and
\begin{equation*}\tag{P1(a)}\label{eq:P1a}
p\rest\gamma\ \force_{\mathbb P_\gamma}\
\Bigl[\,p_\gamma=\emptyset\ \ \vee\ \
\Bigl(p_\gamma= \nm Y^{p}_\gamma \subseteq \Theta\ \wedge\
f^{p}_{\Theta}\bigl[\nm X_\gamma\cap\Theta\bigr]\eqs \nm Y^{p}_\gamma
\Bigr)\,\Bigr],
\end{equation*}
where $\nm Y^{p}_\gamma$ is a nice $\mathbb P_\gamma$-name in the sense of \cite[Chapter IV]{Kunen};
\item $\supp(p):=\{\gamma\ge1: p_\gamma\neq\emptyset\}$ has size $<\kappa$.
\end{itemize}

\item For $p,q\in\mathbb P_\alpha$, $q\le_{\mathbb P_\alpha}p$ if and only
if $\supp(q)\supseteq\supp(p)$ and, upon writing
$\Theta=\hgt(p)$ and $\Theta'=\hgt(q)$, $q_0$ is an
\emph{end-extension} of $p_0$, i.e.
\[
f^{p}=\langle f^{p}_\delta:\delta\le\Theta\rangle\ \subseteq\
f^{q}=\langle f^{q}_\delta:\delta\le\Theta'\rangle ,
\]
and there are $n\in\omega$ and
$\Theta_0=\Theta<\Theta_1<\dots<\Theta_n=\Theta'$ (where $n=0$, i.e.\
$\Theta_0=\Theta=\Theta'$, is allowed)
such that:
\begin{enumerate}[label=(\alph*)]
\item if $\gamma\in\supp(p)$, then
$q\rest\gamma\force_{\mathbb P_\gamma}\ \nm Y^{p}_\gamma=\nm Y^{q}_\gamma\cap\Theta_0$.	
\item $f^{q}_{\Theta_{i+1}}\bigl[[\Theta_i,\Theta_{i+1})\bigr]
=[\Theta_i,\Theta_{i+1})$ for every $i<n$;
\item if $\gamma\in\supp(p)$, then
$q\rest\gamma\force\ (\forall i<n)\
\nm Y^{q}_\gamma\cap[\Theta_i,\Theta_{i+1})
= f^{q}_{\Theta_{i+1}}\bigl[\nm X_\gamma\cap[\Theta_i,\Theta_{i+1})\bigr]$;

\end{enumerate}
\end{enumerate}

We call $\Theta_0<\dots<\Theta_n$ a \emph{witness} to $q \le_{\mathbb P_\alpha} p$, and the
intervals $[\Theta_i,\Theta_{i+1})$ its \emph{blocks}. Note that by (b) the
level $f^{q}_{\Theta_{i+1}}$ restricts to an involution \emph{of} the block
$[\Theta_i,\Theta_{i+1})$. In particular, $f^q_{\Theta_0} \cup \bigcup_{i<n} f^{q}_{\Theta_{i+1}} \rest [\Theta_i,\Theta_{i+1})$ is an involution of $\Theta_n$ under which no point enters or leaves a block.


\begin{observation}\label{obs:basic}
The following facts are immediate from \textup{(P0)--(P2)}:
\begin{enumerate}[label=(\roman*)]
\item $\le_{\mathbb P_\alpha}$ is transitive.
\item For $p,q\in\mathbb P_\alpha$ and $\alpha<\beta$:
$p\le_{\mathbb P_\alpha}q$ iff $p\le_{\mathbb P_\beta}q$. Hence $\mathbb P_\alpha$ is a suborder of $\mathbb P_\beta$ (this uses (P2) only).
\item For $p,q\in\mathbb P_\beta$: $p\le_{\mathbb P_\beta}q$ implies
$p\rest\alpha\le_{\mathbb P_\alpha}q\rest\alpha$ for all $\alpha<\beta$
(the same witness works), but \emph{not} conversely.
\end{enumerate}
\end{observation}

\begin{remark}\label{rem:internalpreserved}
Clause (P2)(c) is consistent with the mod-finite clause \eqref{eq:P1a} for
$q$: if $q\le p$ has witness $\Theta_0<\dots<\Theta_n$, then, block by
block,
$f^{q}_{\Theta'}[\nm X_\gamma\cap[\Theta_i,\Theta_{i+1})]$ differs from
$f^{q}_{\Theta_{i+1}}[\nm X_\gamma\cap[\Theta_i,\Theta_{i+1})]$ by a finite
set (the two functions agree mod finite), and below $\Theta_0$ the images under
$f^{q}_{\Theta'}$ and $f^{p}_{\Theta_0}$ likewise differ finitely; since
$n$ is finite, once again yielding $f^{q}_{\Theta'}[\nm X_\gamma\cap\Theta']\eqs
\nm Y^{q}_\gamma$. Thus finitely many blocks per extension step
cost only finitely much mod-finite error --- the reason (P2) insists on
\emph{finite} witnesses. On the other hand, requiring every nontrivial
extension to have a single block would destroy transitivity.
\end{remark}

\subsection{The recursive construction}\label{sec:recursion}

\subsubsection{Successor step}\label{subsec:successor}
Suppose $\langle\mathbb P_\alpha:\alpha\le\beta\rangle$ is defined and
satisfies (P0), (P1), (P2). Fix a nice $\mathbb P_\beta$-name $\nm X_\beta$ for a
subset of $\kappa$, and define $\mathbb P_{\beta+1}$: $p=\langle p_\rho:
\rho<\beta+1\rangle\in\mathbb P_{\beta+1}$ iff $p\rest\beta\in\mathbb
P_\beta$ and
\[
p\rest\beta\ \force_{\mathbb P_\beta}\
p_\beta=\emptyset\ \vee\
\Bigl(p_\beta=\nm Y^{p}_\beta\ \wedge\
f^{p}_{\hgt(p)}\bigl[\nm X_\beta\cap\hgt(p)\bigr]\eqs\nm Y^{p}_\beta\Bigr),
\]
with the order given by (P2).

Clearly $\mathbb P_\beta\subseteq\mathbb P_{\beta+1}$ as a suborder
(identify $p\in\mathbb P_\beta$ with $p{}^\frown \emptyset$; a suborder
since we have (P2) for $\mathbb{P}_\beta$). For
(P0), by Observation \ref{obs:basic} (iii), we know that any two incompatible conditions in $\mathbb{P}_\beta$ are incompatible in $\mathbb{P}_{\beta+1}$. It remains to see that maximal antichains of $\mathbb P_\beta$ stay
predense. Let $A\subseteq\mathbb P_\beta$ be a maximal antichain,
$p\in\mathbb P_{\beta+1}$, and pick $a\in A$ and
$q\le_{\mathbb P_\beta}a,\,p\rest\beta$. Let
$\Theta_0<\Theta_1<\dots<\Theta_m$ witness $q\le p\rest\beta$. Extend $q$ to
$q'\in\mathbb P_{\beta+1}$ by defining the $\beta$-th coordinate
\[
q'_\beta:=\nm Y^{p}_\beta\cup
\bigcup_{i<m}f^{q}_{\Theta_{i+1}}\bigl[\nm X_\beta\cap[\Theta_i,
\Theta_{i+1})\bigr]
\]
(as a $\mathbb P_\beta$-name). Then $q'\in\mathbb P_{\beta+1}$ --- the
internal clause (P1) for $q'_\beta$ is forced by
Remark~\ref{rem:internalpreserved} --- and $\Theta_0,\dots,\Theta_m$ witness
$q'\le_{\mathbb P_{\beta+1}}p$, while $q'\le q\le a$. Hence $A$ remains
predense, and therefore $\mathbb P_\beta\lessdot\mathbb P_{\beta+1}$.

\subsubsection{Limit step}\label{subsec:limit}
Suppose $\langle\mathbb P_\alpha:\alpha<\beta\rangle$ is defined and
satisfies (P0), (P1), (P2). Let the elements of $\mathbb P_\beta$ be the
${<}\kappa$-support limit of $\langle \mathbb{P}_{\alpha}: \alpha<\beta\rangle$, and the ordering of $\mathbb{P}_\beta$ be as described in (P2). Let us stress that it is the order of $\mathbb{P}_\beta$ that makes it not merely a standard $<\kappa$-support iteration.
 Clearly each $\mathbb P_\alpha\subseteq\mathbb P_\beta$ is a
suborder and incompatibility relations are preserved. It remains to check the predensity of maximal antichains, i.e.\
$\mathbb P_\alpha\lessdot\mathbb P_\beta$ for $\alpha<\beta$.

Let $A\subseteq\mathbb P_\alpha$ be a maximal antichain and
$p\in\mathbb P_\beta\setminus\mathbb P_\alpha$. Pick $a\in A$ and
$q\le_{\mathbb P_\alpha}a,\,p\rest\alpha$, and let
$\Theta_0=\hgt(p)=\hgt(p\rest\alpha)<\Theta_1<\dots<\Theta_m=\hgt(q)$
witness $q\le_{\mathbb P_\alpha}p\rest\alpha$. If $\Theta_0=\hgt(q)$,
i.e.\ $m=0$, then
$q{}^\frown\langle p_\rho:\rho\in[\alpha,\beta)\rangle\le_{\mathbb
P_\beta}q,\,p$ by the very definition of $\mathbb P_\beta$, and we are
done. We extend $q$ to
$q'\le_{\mathbb P_\beta}p$ with $q'\rest\alpha=q$, coordinate by coordinate
along $\supp(p)\setminus\alpha=\{\xi_i:i<\otp(\supp(p)\setminus\alpha)\}$
(enumerated increasingly): for each such $\xi=\xi_i$ put
\begin{equation}\label{eq:tailext}
q'_\xi:= \nm Y^{p}_\xi\cup
\bigcup_{j<m}f^{q}_{\Theta_{j+1}}\bigl[\nm X_\xi\cap[\Theta_j,
\Theta_{j+1})\bigr],
\end{equation}
and $q'_\rho:=\emptyset$ for the remaining $\rho\in\beta\setminus
(\alpha\cup\supp(p))$ (pedantically, $q'_\xi$ is a $\mathbb
P_\xi$-name such that $1_{\mathbb P_\xi}$ forces
\eqref{eq:tailext}).

We prove simultaneously by induction on $i$ that
$q'\rest(\xi_i+1)\in\mathbb P_{\xi_i+1}$ and
$q'\rest(\xi_i+1)\le p\rest(\xi_i+1)$, with the \emph{same} witness
$\Theta_0,\dots,\Theta_m$ throughout. At the $i$-th step the induction
hypothesis first gives $q'\rest\xi_i\in\mathbb P_{\xi_i}$ and
$q'\rest\xi_i\le p\rest\xi_i$. Consequently:
\begin{itemize}
\item $q'\rest\xi_i\force_{\mathbb P_{\xi_i}}$
``$f^{q}_{\Theta_m}[\nm X_{\xi_i}\cap\Theta_m]
\eqs q'_{\xi_i}=\nm Y^{q'}_{\xi_i}$''.
Indeed, $q'\rest\xi_i\le p\rest\xi_i$, so clause
\eqref{eq:P1a} for $p$ gives
$f^{p}_{\Theta_0}[\nm X_{\xi_i}\cap\Theta_0]
\eqs\nm Y^{p}_{\xi_i}$.
Coherence gives
$f^{q}_{\Theta_m}\rest\Theta_0\eqs f^{p}_{\Theta_0}$
and
$f^{q}_{\Theta_m}\rest\Theta_{j+1}
\eqs f^{q}_{\Theta_{j+1}}$ for every $j<m$.
Thus \eqref{eq:tailext}, together with the finiteness of the
witness, gives the displayed equality modulo finite
(cf.\ Remark~\ref{rem:internalpreserved}).
Hence $q'\rest(\xi_i+1)$ satisfies (P1).
\item (P2)(b) for the witness is a statement only about the first coordinate $f^{q}$, which holds
because it held for $q\le p\rest\alpha$.
\item (P2)(a),(c) at $\xi_i$ hold by the very definition of
$q'_{\xi_i}$.
\end{itemize}
At limit stages of the enumeration the same simultaneous induction passes
through unions of the coordinates. Thus $q'\in\mathbb P_\beta$ and
$q'\le_{\mathbb P_\beta}p$; moreover $q'\le q\le a$. Hence
$\mathbb P_\alpha\lessdot\mathbb P_\beta$, completing the verification of (P0).

\section{Main properties of the forcing}\label{sec:mainlemma}

\begin{lemma}\label{lem:mainlemma}
Let $\beta\le\kappa^{+}$. 
\begin{enumerate}[label=\textup{\arabic*)},start=0]
\item For $1\le\alpha<\beta\leq \kappa^{+}$ and $\Theta'<\kappa$, the set
$D_{\alpha,\Theta'}=\{p\in\mathbb P_{\beta}:\alpha\in\supp(p),\
\hgt(p)\ge\Theta'\}$ is dense open.
\item $\mathbb P_\beta$ is $\kappa$-strategically closed.
\item Let $\mathbb P^{0}_\beta$ consist of all $p\in\mathbb P_\beta$
such that, for every $1\le\gamma<\beta$, either $p_\gamma=\emptyset$
or $p\rest\gamma$ decides $p_\gamma$; that is,
$p\rest\gamma\force p_\gamma=\check A$ for some
$A\in[\kappa]^{<\kappa}$. Then $\mathbb P^{0}_\beta$ is dense in
$\mathbb P_\beta$.

\item $\mathbb P_{\kappa^{+}}$ is $\kappa^{+}$-cc.
\end{enumerate}

Let $G\subseteq\mathbb P_{\kappa^{+}}$ be
generic over $V$. Write $f^{G}_\delta$ for the generic tower
$\bigcup\{f^{p}_\delta:p\in G,\ \hgt(p)\ge\delta\}$.
\begin{enumerate}[label=\textup{\arabic*)},start=4]
\item $G$ defines an almost trivial automorphism $\FG$ of
$\Pw(\kappa)/\Fin$. Moreover, $\FG\circ\FG=\mathrm{id}$, and for every
$\beta<\kappa$, its restriction to $\Pw(\beta)/\Fin$ is induced by
$f^{G}_\beta$.
 \item In $V[G]$, suppose $E\colon X\to\kappa$ for some
 $X\in[\kappa]^{\kappa}$. Then
$\exists\delta<\kappa\
E\rest(X\cap \delta)\not\eqs f^{G}_\delta\rest( X\cap\delta)$.
\end{enumerate}

\end{lemma}

The proof occupies \S\S\ref{sec:density}--\ref{sec:nontrivial}.

\subsection{Density of heights and supports: item 0)}\label{sec:density}

\begin{proof}[Proof of \textup{0)}]
Let $\beta\leq \kappa^+$, $p\in\mathbb P_{\beta}$, $\Theta=\hgt(p)$, and let
$\Theta'>\Theta$, $1\le\alpha<\beta$ be given. Define $q\le_{\mathbb{P}_\beta} p$ as follows:
\[
f^{q}_\delta:=f^{p}_{\Theta}\cup\mathrm{id}_{[\Theta,\delta)}
\qquad(\Theta<\delta\le\Theta'),
\]
so the tower is extended by the identity; for $\gamma\in\supp(p)$ let
\[
q_\gamma:=\nm Y^{p}_\gamma\cup\bigl(\nm X_\gamma\cap
[\Theta,\Theta')\bigr),
\]
and, if $\alpha\notin\supp(p)$, activate the coordinate $\alpha$ by
\[
q_\alpha:= f^{q}_{\Theta'}\bigl[\nm X_\alpha\cap
\Theta'\bigr].
\]
Since $f^{q}_{\Theta'}\rest[\Theta,\Theta')=\mathrm{id}$, the pair
$\Theta_0=\Theta<\Theta_1=\Theta'$ witnesses $q\le p$: (a) is clear, (b)
holds because the identity fixes $\nm X_\gamma\cap[\Theta,\Theta')$
pointwise, and (c) by construction. Hence $q\in D_{\alpha,\Theta'}$ and $D_{\alpha,\Theta'}$ is
dense open.
\end{proof}

\subsection{Strategic closure: item 1)}\label{sec:closure}

\begin{claim}\label{claim:fusion}
Let $0<\rho<\kappa$ and let
$\vec p=\langle p_i:i<\rho\rangle$ be decreasing in $\mathbb P_\beta$.
Suppose that, for all $i<j<\rho$, the following statement holds:
\begin{equation*}
(\bigstar_{i,j})\qquad
\begin{array}{rl}
& p_j\le p_i\text{ is witnessed by the single pair }
\hgt(p_i)\le\hgt(p_j),\\
\textrm{and } & f^{p_j}_{\hgt(p_j)}\rest\hgt(p_i)=f^{p_i}_{\hgt(p_i)}.
\end{array}
\end{equation*}
Then:
\begin{enumerate}[label=\textup{(C\arabic*)}]
\item there is a lower bound $p_\rho$ for $\vec{p}$ such that
$(\bigstar_{i,\rho})$ holds for all $i<\rho$. Moreover, there is a unique
such $p_\rho$ whose support is obtained by
 the union of those of the $p_i$ and 
\[
\textrm{that satisfies } \hgt(p_\rho)=\sup_{i<\rho}\hgt(p_i).
\]
We call it the \emph{$\bigstar$-limit} of $\vec p$.
\item \label{c2} if moreover $\rho=\sigma+1$, i.e.\ $\vec p$ has a last element
$p_\sigma$, and $q\le p_\sigma$, then there is a lower bound
$p_\rho\le q$ satisfying $(\bigstar_{i,\rho})$ for all $i<\rho$.
\end{enumerate}
\end{claim}

\begin{proof}[Proof of \textup{(C1)}]
If $\rho=\sigma+1$ is a successor, put $p_\rho:=p_\sigma$; the trivial
witness shows that this is a suitable lower bound. Assume that $\rho$ is
a limit ordinal, put
$\Theta_i:=\hgt(p_i)$ and $\Theta_\rho:=\sup_{i<\rho}\Theta_i$, and define
\[
f^*:=\bigcup_{i<\rho}f^{p_i}_{\Theta_i}.
\]
The second clause of $(\bigstar_{i,j})$ says that the functions in this
union form an increasing chain. Hence $f^*$ is a function on
$\Theta_\rho$, and $f^*\rest\Theta_i=f^{p_i}_{\Theta_i}$ for every
$i<\rho$. Moreover, each interval $[\Theta_i,\Theta_{i+1})$ is
$f^*$-invariant by (P2)(b), so $f^*$ is an involution of $\Theta_\rho$.
For $\delta\le\Theta_i$, tower coherence inside $p_i$ gives
\[
f^*\rest\delta=f^{p_i}_{\Theta_i}\rest\delta\eqs f^{p_i}_\delta,
\]
so putting  $f^*$ on top results in a coherent system.

Let $\supp(p_\rho):=\bigcup_{i<\rho}\supp(p_i)$.
Define its tower by
\[
f^{p_\rho}_\delta:=
\begin{cases}
f^{p_i}_\delta,
 & \delta<\Theta_\rho\text{ and }i<\rho
   \text{ is any index with }\delta\le\Theta_i,\\
f^*,
 & \delta=\Theta_\rho.
\end{cases}
\]
The first case is independent of the choice of $i$, since the
tower parts form an increasing chain under end-extension.
For every coordinate in its support, define
\[
\nm Y^{p_\rho}_\gamma:=\bigcup_{i<\rho}\nm Y^{p_i}_\gamma.
\]
The support still has size ${<}\kappa$. A simultaneous inductive argument (similar to Section \ref{subsec:successor}) yields for each $\xi$ in the support that $p_\rho \rest \xi$ is a condition, $p_\rho \rest \xi \leq p_i \rest \xi$ witnessed by a single block, and $p_\rho \rest \xi$ forces $\nm Y_\xi^{p_\rho} \eqs f^*[\nm X_\xi \cap \Theta_\rho]$.  (E.g.\ to check \eqref{eq:P1a}, fix
$\gamma\in\supp(p_\rho)$ and choose $i_0<\rho$ with
$\gamma\in\supp(p_{i_0})$. For every $i\ge i_0$, (P2)(c) for
$p_{i+1}\le p_i$, together with the pointwise equality of $f^*$ and
$f^{p_{i+1}}_{\Theta_{i+1}}$ below $\Theta_{i+1}$, gives
\[
(p_{i+1} \Vdash ) \quad f^*\bigl[\nm X_\gamma\cap[\Theta_i,\Theta_{i+1})\bigr]
=\nm Y^{p_\rho}_\gamma\cap[\Theta_i,\Theta_{i+1}).
\]
Their union is the required exact equality above $\Theta_{i_0}$; below
$\Theta_{i_0}$, \eqref{eq:P1a} for $p_{i_0}$ gives equality modulo finite.) This also shows that $(\bigstar_{i,\rho})$ holds for every
$i<\rho$. The displayed unions also give the asserted uniqueness.
\end{proof}

Given $q\le p$ with witness $\Theta_0<\dots<\Theta_n$, define the
\emph{diagonal} of the witness to be the involution $s_0$ of $\hgt(q)$
given by
\begin{equation}\label{eq:diagonal}
s_0\rest\Theta_0:=f^{q}_{\Theta_0}=f^p_{\Theta_0},\qquad
s_0\rest[\Theta_j,\Theta_{j+1}):=
f^{q}_{\Theta_{j+1}}\rest[\Theta_j,\Theta_{j+1})\quad(j<n).
\end{equation}
Each piece is an involution of its interval by (P2)(b).

\begin{lemma}[Coarsening]\label{lem:coarsening}
Let $q\le p$ in $\mathbb P_\beta$ have witness
$\Theta_0<\dots<\Theta_n$, with diagonal $s_0$, and let $s$ be an
involution of $\hgt(q)$ such that
\begin{enumerate}[label=\textup{(s\arabic*)}]
\item\label{s:diag} $s\rest[\Theta_0,\hgt(q))=s_0\rest[\Theta_0,\hgt(q))$;
\item\label{s:modfin} $s\rest\Theta_0\eqs f^{q}_{\Theta_0}$;

\end{enumerate}
Then for every $\Theta''>\hgt(q)$ there is $r\le q$ such that
\[
\hgt(r)=\Theta'',\qquad f^{r}_{\Theta''}\rest\hgt(q)=s,
\]
and $r\le p$ is witnessed by the single pair $\Theta_0<\Theta''$.
\end{lemma}
(Note that $s[\Theta_0]=\Theta_0$ holds, too, because $s_0\rest[\Theta_0,\hgt(q))$ is onto $[\Theta_0,\hgt(q))$.)
\begin{proof}
Define $r$ by
\begin{align*}
r_0&:=\langle f^{q}_\delta:\delta\le\hgt(q)\rangle{}^\frown
\langle s\cup\mathrm{id}_{[\hgt(q),\delta)}:
\hgt(q)<\delta\le\Theta''\rangle,\\
r_\gamma&:=\nm Y^{q}_\gamma\cup
\bigl(\nm X_\gamma\cap[\hgt(q),\Theta'')\bigr)
\qquad(\gamma\in\supp(q)),\\
r_\gamma&:=\emptyset\qquad
(\gamma\in\beta\setminus(\supp(q)\cup\{0\})).
\end{align*}
It is straightforward to check that the displayed tower (i.e.\ $r_0$) is coherent. Indeed,
the new levels agree pointwise above $\hgt(q)$, while
$s\eqs f^q_{\hgt(q)}$: on each old witness block this follows from
\ref{s:diag} and tower coherence inside $q$, and below $\Theta_0$ it follows
from \ref{s:modfin}; only finitely many errors are involved.

We verify simultaneously, by induction on $\xi\le\beta$, that
$r\rest\xi$ is a condition, that it extends both $q\rest\xi$ and
$p\rest\xi$ with the indicated single-pair witnesses, and that, whenever
$\gamma<\xi$ belongs to $\supp(q)$ (for \eqref{eq:coarse-old}, to
$\supp(p)$),
\begin{align}
r\rest\gamma&\force
f^r_{\Theta''}\bigl[\nm X_\gamma\cap[\Theta_0,\hgt(q))\bigr]
=\nm Y^r_\gamma\cap[\Theta_0,\hgt(q)),\label{eq:coarse-old}\\
r\rest\gamma&\force
f^r_{\Theta''}\bigl[\nm X_\gamma\cap[\hgt(q),\Theta'')\bigr]
=\nm Y^r_\gamma\cap[\hgt(q),\Theta'').\label{eq:coarse-new}
\end{align}
There is nothing to do at a coordinate outside $\supp(q)$. At a coordinate
$\gamma\in\supp(q)$, \eqref{eq:coarse-new} follows from the definition of
$r_\gamma$, since the new part of $f^r_{\Theta''}$ is the identity. On an
old witness block $B_j=[\Theta_j,\Theta_{j+1})$, (P2)(c) for $q\le p$
(applicable as $\gamma\in\supp(p)$) and \ref{s:diag} give
\[
\nm Y^r_\gamma\cap B_j=\nm Y^q_\gamma\cap B_j
=f^q_{\Theta_{j+1}}[\nm X_\gamma\cap B_j]
=f^r_{\Theta''}[\nm X_\gamma\cap B_j].
\]
 They give (P2)(c) for both extensions;  (P2)(a) follows directly from the definition of $r_\gamma$ and the other clauses are trivial. Finally,
\eqref{eq:P1a} for $r$ follows from \eqref{eq:P1a} for $q$, the relation
$s\eqs f^q_{\hgt(q)}$, and the exact equality on the new block. This proves
the induction and hence all the asserted properties of $r$.
\end{proof}

\begin{proof}[Proof of \textup{(C2)}]
Let $\rho=\sigma+1$, let $q\le p_\sigma$, and let $s_0$ be the diagonal
of a witness for this extension. Apply Lemma~\ref{lem:coarsening} with
$s=s_0$ and any $\Theta''>\hgt(q)$, and call the resulting condition
$p_\rho$. Since towers end-extend,
\begin{equation}\label{eq:diagbase}
f^{p_\rho}_{\Theta''}\rest\hgt(p_\sigma)
=s_0\rest\hgt(p_\sigma)
=f^{p_\sigma}_{\hgt(p_\sigma)}.
\end{equation}
Thus $(\bigstar_{\sigma,\rho})$ holds. For $i<\sigma$, $(\bigstar_{\sigma,\rho})$ and $(\bigstar_{i,\sigma})$ give the required
pointwise agreement below $\hgt(p_i)$. Moreover,
$p_\rho\le p_i$ is witnessed by the single pair
$\hgt(p_i)\le\Theta''$: full coherence and 3 witnesses ($\hgt(p_i), \hgt(p_\sigma), \hgt(p_\rho)$) together imply the two endpoints are enough.
\end{proof}

\begin{proof}[Proof of \textup{1)}]
We describe a strategy for Player II in the descending game of length
$\kappa$ on $\mathbb P_\beta$. She maintains the invariant that her own
moves, in their natural order, satisfy $(\bigstar_{i,j})$ of
Claim~\ref{claim:fusion}. After a move by Player I, clause (C2) supplies an
extension which restores this invariant. At a limit turn, Player II plays
the $\bigstar$-limit (of her moves) supplied by (C1). These moves are available at every
stage below $\kappa$: regularity keeps all heights below $\kappa$, and the
union of fewer than $\kappa$ supports of size ${<}\kappa$ again has size
${<}\kappa$. Thus Player II survives the game of length $\kappa$.
\end{proof}


Consequently $\mathbb P_\beta$ is ${<}\kappa$-distributive: it adds no new
sequences of length $<\kappa$ (so no new bounded subsets of $\kappa$), thus cardinals
and cofinalities $\le\kappa$ as well as the inaccessibility of $\kappa$ are preserved.

\subsection{Decided conditions are dense: item 2)}\label{sec:decided}

We first record a simple observation.

\begin{observation}\label{obs:decided-limit}
Let $\langle p_i:i\le\delta\rangle\subseteq\mathbb P_\alpha$ be decreasing,
and suppose that $\hgt(p_\delta)$ is a limit ordinal and that, for every $\gamma\in\supp(p_\delta)$ and every
$\Theta <\hgt(p_\delta)$, there is an $i\le\delta$ such that
$p_i\rest\gamma$ decides
$\nm Y^{p_\delta}_\gamma\cap\Theta$. Then
$p_\delta\in\mathbb P^0_\alpha$.
\end{observation}

Indeed, for each $\gamma\in\supp(p_\delta)$, the condition
$p_\delta\rest\gamma$ decides every initial segment of
$\nm Y^{p_\delta}_\gamma$. Their union is a ground-model subset of
$\hgt(p_\delta)$.

\begin{proof}[Proof of \textup{2)}]
By induction on $\beta\leq \kappa^+$.

The case $\beta=0$ is immediate.

\emph{Successor $\beta+1$, assuming $\mathbb P^{0}_\beta$ dense.} Let
$p\in\mathbb P_{\beta+1}$. If $\beta\notin\supp(p)$, then the induction hypothesis applies; so assume $\beta\in\supp(p)$ and put $\Theta:=\hgt(p)$.

Put $r_0:=p$. Recursively, having chosen $r_i$, write
$h_i:=\hgt(r_i)$. Using ${<}\kappa$-distributivity and the density of
$\mathbb P^0_\beta$, choose $u_i\le r_i\rest\beta$ in
$\mathbb P^0_\beta$ so that $u_i$ decides
$\nm X_\beta\cap h_i$ and $\nm Y^{r_i}_\beta$ (then it follows from $u_i \in \mathbb{P}_\beta^0$ that for every
$\gamma\in\supp(r_i)\cap\beta$, $u_i\rest\gamma$ decides
$\nm Y^{r_i}_\gamma$).

Lift $u_i$ through the last coordinate as in the successor construction of \ref{sec:recursion},
and call the resulting condition $\widetilde r_i\le r_i$; thus
$\widetilde r_i\rest\beta=u_i$. Apply Lemma~\ref{lem:coarsening} in
$\mathbb P_{\beta+1}$ to $\widetilde r_i\le r_i$, using the diagonal of
a witness for this extension. This gives
$r_{i+1}\le\widetilde r_i$ such that $r_{i+1}\le r_i$ is witnessed by the
single pair
\[
h_i<h_{i+1}:=\hgt(r_{i+1}),
\qquad
f^{r_{i+1}}_{h_{i+1}}\rest h_i=f^{r_i}_{h_i}.
\]
Consequently, its last coordinate is
\[
\nm Y^{r_{i+1}}_\beta
=\nm Y^{r_i}_\beta\cup
f^{r_{i+1}}_{h_{i+1}}
 \bigl[\nm X_\beta\cap[h_i,h_{i+1})\bigr].
\]
In particular, $(\bigstar_{i,j})$ holds for all $i<j<\omega$.

Let $r_\omega$ be the $\bigstar$-limit of
$\langle r_i:i<\omega\rangle$. Apply Observation~\ref{obs:decided-limit}
to the decreasing sequence obtained by interleaving the $r_i$ and the
$\widetilde r_i$ and appending $r_\omega$. Indeed, if
$\gamma\in\supp(r_\omega)$ and $\Xi<\hgt(r_\omega)$, choose $i$ so large
that $\gamma\in\supp(r_i)$ and $\Xi \leq h_i$. Then
$r_{i+1}\rest\gamma$ decides $\nm Y^{r_i}_\gamma\cap\Xi$, and
$r_\omega\rest\gamma$ forces
$\nm Y^{r_\omega}_\gamma\cap\Xi=\nm Y^{r_i}_\gamma\cap\Xi$.
Hence $r_\omega\in\mathbb P^0_{\beta+1}$, and $r_\omega\le p$.

\emph{Limit $\beta$.} If $\cf(\beta)\ge\kappa$, then by the
${<}\kappa$-support limit every $p\in\mathbb P_\beta$ lies in some
$\mathbb P_\alpha$, $\alpha<\beta$, and we are done by induction. So
assume $\nu:=\cf(\beta)<\kappa$ and fix a continuous increasing sequence
$\langle\alpha_i:i<\nu\rangle$ cofinal in $\beta$. Given
$p\in\mathbb P_\beta$, recursively construct
$p\ge q_0 \ge q_0' \ge q_1\ge\cdots \geq q_i \geq q_i' \geq q_{i+1} \geq \cdots$ ($i<\nu$) such that
\[
q'_i\rest\alpha_i\in\mathbb P^{0}_{\alpha_i}\quad(i<\nu),
\qquad\text{and}\qquad
\langle q_i:i<\nu\rangle\ \text{satisfies}\
(\bigstar_{i,j})\ \text{of Claim~\ref{claim:fusion} for all }i<j<\nu.
\]
Put $q_0:=p$. Given $q_i$, apply
the density of $\mathbb P^0_{\alpha_i}$ in $\mathbb P_{\alpha_i}$ to $q_i\rest\alpha_i$, and lift the resulting condition
through the tail exactly as in \S\ref{subsec:limit} to obtain $q_i'\le q_i$. Then apply Lemma \ref{lem:coarsening} to $q_i' \le q_i$;
by (C2) the resulting $q_{i+1}$ satisfies the required $\bigstar$ relations with all preceding $q_j$.

At a nonzero limit stage $i<\nu$, let $q_i$ be the $\bigstar$-limit of
$\langle q_j:j<i\rangle$. The invariant is preserved by (C1).

Finally, let $q$ be the $\bigstar$-limit of
$\langle q_i:i<\nu\rangle$, and use Observation \ref{obs:decided-limit} to conclude that this limit is in $\mathbb{P}^0_\beta$.
\end{proof}

\subsection{The chain condition: item 3)}\label{sec:cc}

\begin{proof}[Proof of \textup{3)}]
Let $\{p^{\iota}:\iota<\kappa^{+}\}\subseteq\mathbb P_{\kappa^{+}}$. By
item 2) we may assume each $p^{\iota}\in\mathbb P^{0}_{\kappa^{+}}$: each
coordinate is either $\emptyset$ or $y$ with
$y\in\Pw(\hgt(p^\iota))\cap V$. There are $\kappa$ many possible heights
and, since $\kappa$ is inaccessible, at most
$2^{<\kappa}=\kappa$ many tower parts; refine to constant height
$\Theta$ and constant tower part. Supports have size $<\kappa$, and
$\alpha^{<\kappa}<\kappa^{+}$ for
$\alpha<\kappa^{+}$ (by $2^{\kappa}=\kappa^{+}$ and inaccessibility), so
the $\Delta$-system lemma applies: refine to supports forming a
$\Delta$-system with root $R$, and, as there are at most
$\kappa^{<\kappa}=\kappa$ possibilities for the data indexed by $R$,
further so that $p^{\iota}\rest R=p^{\iota'}\rest R$ for all
$\iota,\iota'<\kappa^{+}$. Any two members $p,q$ of the refined family are
compatible: let $r$ have the common tower part and
$r_\gamma:=p_\gamma\cup q_\gamma$ for every $\gamma$.
In particular, every nonempty coordinate of $r$ is literally a coordinate
of either $p$ or $q$, hence it is a common extension once we verify that it is a
condition.
By induction on $\gamma\le\kappa^{+}$,
\[
r\rest\gamma\in\mathbb P_\gamma,\qquad
r\rest\gamma\le p\rest\gamma\quad\text{and}\quad
r\rest\gamma\le q\rest\gamma\quad\text{with the trivial witness}:
\]
the clauses of (P1) at a coordinate $\gamma$ are forced by
$r\rest\gamma$, since by the inductive hypothesis $r\rest\gamma$ extends
both $p\rest\gamma$, $q\rest\gamma$ and the one that has $\gamma$ in its support forces (P1). Hence there is
no antichain of size $\kappa^{+}$.
\end{proof}

\begin{corollary}\label{cor:preservation}
$\mathbb P_{\kappa^{+}}$ preserves cofinalities and cardinals;
$\kappa$ remains inaccessible; and in $V^{\mathbb P_{\kappa^{+}}}$,
$2^{\kappa}=\kappa^{+}$. Moreover every subset of $\kappa$ of the
extension has a nice $\mathbb P_\alpha$-name for some $\alpha<\kappa^{+}$,
and for each $\alpha$ there are at most $\kappa^{+}$ such names; hence the
bookkeeping in the proof of Theorem~\ref{theorem: main} can be realized, and in
$V[G]$ every $X\subseteq\kappa$ equals $\nm X_\gamma[G]$ for cofinally
many $\gamma<\kappa^{+}$, each of which lies in $\supp(p)$ for some $p\in
G$ by item \textup{0)} and genericity. \qed
\end{corollary}

\subsection{The generic automorphism: item 4)}\label{sec:auto}

Let $G\subseteq\mathbb P_{\kappa^{+}}$ be generic. For
$\gamma\in[1,\kappa^{+})$ set $X_\gamma:=\nm X_\gamma[G\rest\gamma]$ and
\[
Y_\gamma:=\bigcup\bigl\{\nm Y^{p}_\gamma[G\rest\gamma]:
p\in G,\ \gamma\in\supp(p)\bigr\}.
\]
Note that the sets in this union are end-extensions of one another: if
$p,q\in G$, $\gamma\in\supp(p)\cap\supp(q)$ and $\hgt(p)\le\hgt(q)$, then
$\nm Y^{q}_\gamma[G\rest\gamma]\cap\hgt(p)=\nm Y^{p}_\gamma[G\rest\gamma]$
--- this follows from (P2)(a), e.g.\ by picking $r\in G$ with $r\le p,q$
and applying (P2)(a) to $r\le p$ and to $r\le q$. By item 0) the heights of conditions in $G$
with $\gamma\in\supp$ are cofinal in $\kappa$, so $Y_\gamma$ is defined and
$Y_\gamma\cap\hgt(p)=\nm Y^{p}_\gamma[G\rest\gamma]$ for such $p$. Define
\[
\FG\bigl([X_\gamma]\bigr):=[Y_\gamma].
\]
By Corollary~\ref{cor:preservation} every $X\in\Pw(\kappa)^{V[G]}$ is of the
form $X_\gamma$, so $\FG$ is defined on all of $\Pw(\kappa)/\Fin$ once we
verify it is well defined on classes.

\begin{lemma}[Coherence]\label{lem:coherence}
For every $\gamma$ as above and every $\delta<\kappa$:
\[
Y_\gamma\cap\delta\ \eqs\ f^{G}_\delta\bigl[X_\gamma\cap\delta\bigr].
\]
\end{lemma}

\begin{proof}
Pick $p\in G$ with $\gamma\in\supp(p)$ and $\Theta:=\hgt(p)\ge\delta$
(item 0)). The internal clause (P1) of $p$, interpreted in $G$, gives
$Y_\gamma\cap\Theta\eqs f^{G}_{\Theta}[X_\gamma\cap\Theta]$. Now apply
Lemma~\ref{lem:localimage}.
\end{proof}

\begin{proof}[Proof of \textup{4)}, well-definedness]
Suppose $p\in G$ and $p\force\nm X_\alpha\eqs\nm X_{\alpha'}$
($\alpha<\alpha'<\kappa^{+}$). By item 0) and by extending $p$ inside $G$
we may assume $\alpha,\alpha'\in\supp(p)$ and, by possibly further
extending,
\[
p\force\ \nm X_\alpha\,\triangle\,\nm X_{\alpha'}\subseteq\hgt(p);
\quad\text{therefore}\quad
p\force\ \nm X_\alpha\setminus\hgt(p)=\nm X_{\alpha'}\setminus\hgt(p).
\]
We claim that $Y_\alpha\eqs Y_{\alpha'}$ in $V[G]$; indeed
$Y_\alpha\setminus\hgt(p)=Y_{\alpha'}\setminus\hgt(p)$. To see the latter
equation it is clearly enough to argue that
\begin{equation}\label{eq:wd}
\bigl(Y_\alpha\setminus\hgt(p)\bigr)\cap\Theta
=\bigl(Y_{\alpha'}\setminus\hgt(p)\bigr)\cap\Theta
\qquad\text{for every }\Theta<\kappa.
\end{equation}
Fix $\Theta<\kappa$, and fix $q\in G$ with $q\le p$ and
$\hgt(q)\ge\Theta$ (item 0) and genericity). Let
$\Theta_0<\dots<\Theta_n$ witness $q\le p$, with blocks
$B_i=[\Theta_i,\Theta_{i+1})$ and levels $f_i=f^{q}_{\Theta_{i+1}}$.
Since $B_i\cap\hgt(p)=\emptyset$, clause (P2)(c), applied at $\alpha$ and
at $\alpha'$ \emph{with the same blocks and the same levels}, gives that
$q$ forces, for every $i<n$,
\[
\nm Y^{q}_\alpha\cap B_i
= f_i\bigl[\nm X_\alpha\cap B_i\bigr]
= f_i\bigl[\nm X_{\alpha'}\cap B_i\bigr]
= \nm Y^{q}_{\alpha'}\cap B_i .
\]
As the blocks partition $[\hgt(p),\hgt(q))$, evaluating in $G$ and (using
$Y_\alpha\cap\hgt(q)=\nm Y^{q}_\alpha[G\rest\alpha]$, and likewise at $\alpha'$):
\[
\bigl(Y_\alpha\setminus\hgt(p)\bigr)\cap\Theta
=\bigl(\nm Y^{q}_\alpha[G\rest\alpha]\setminus\hgt(p)\bigr)\cap\Theta
=\bigl(\nm Y^{q}_{\alpha'}[G\rest\alpha']\setminus\hgt(p)\bigr)\cap\Theta
=\bigl(Y_{\alpha'}\setminus\hgt(p)\bigr)\cap\Theta ,
\]
which is \eqref{eq:wd}. Below $\hgt(p)$, the internal clauses (P1),
evaluated in $G$, give
\[
Y_\alpha\cap\hgt(p)\ \eqs\ f^{G}_{\hgt(p)}\bigl[X_\alpha\cap\hgt(p)\bigr]
\ \eqs\ f^{G}_{\hgt(p)}\bigl[X_{\alpha'}\cap\hgt(p)\bigr]
\ \eqs\ Y_{\alpha'}\cap\hgt(p),
\]
the middle $\eqs$ since $X_\alpha\cap\hgt(p)\eqs X_{\alpha'}\cap\hgt(p)$
and images of $\eqs$-equal sets under an injection are $\eqs$. Hence
$Y_\alpha\eqs Y_{\alpha'}$, as claimed whenever
$p$ belongs to the generic filter.
\end{proof}

\begin{proof}[Proof of \textup{4)}, Boolean structure]
Let $X=X_\alpha$, $X'=X_{\alpha'}$ in $V[G]$ and let $\beta,\beta'$ be
bookkeeping coordinates with $X_\beta=X\cup X'$ and $X_{\beta'}=X\cap X'$.
Fix $p\in G$ with $\alpha,\alpha',\beta,\beta'\in\supp(p)$,
$\Theta=\hgt(p)$ and that forces the above identities. Then 
\[
p\force\ \nm X_\beta\setminus\Theta
=(\nm X_\alpha\cup\nm X_{\alpha'})\setminus\Theta,
\qquad
p\force\ \nm X_{\beta'}\setminus\Theta
=(\nm X_\alpha\cap\nm X_{\alpha'})\setminus\Theta .
\]
Now for $q\in G$ with $q\le p$ the above identities and (P2)(c) at $\beta$, $\alpha$,
$\alpha'$ give that $q$ forces
\[
\nm Y^{q}_\beta\cap [\Theta, \hgt(q))=
(\nm Y^{q}_\alpha\cup\nm Y^{q}_{\alpha'})\cap [\Theta, \hgt(q)) .
\]
Exactly as in the proof of well-definedness (running over all
$\Theta'<\kappa$ and all $q\in G$ of height $\ge\Theta'$), this yields
$(Y_\alpha\cup Y_{\alpha'})\setminus\Theta=Y_\beta\setminus\Theta$ in
$V[G]$. On the other hand, below $\Theta$ the fact that
$q$ is a condition together with $p \Vdash \nm X_\beta = \nm X_\alpha \cup \nm X_{\alpha'}$ give agreement mod finite. Hence
$\FG([X\cup X'])=\FG([X])\vee\FG([X'])$; the case of $\wedge$ is
identical; and for complements one uses additionally that
by (P2)(b)  $f^q_{\hgt(q)}$ restricts to a self-bijection of the block $[\Theta, \hgt(q))$, so if $q \leq p$ is witnessed by $\Theta < \hgt(q)$, then
\[
f^q_{\hgt(q)}\bigl[ [\Theta, \hgt(q)) \setminus \nm X_\alpha\bigr]
=[\Theta, \hgt(q)) \setminus f^q_{\hgt(q)}\bigl[\nm X_\alpha \bigr].
\]
Also $\FG([\emptyset])=[\emptyset]$ and $\FG([\kappa])=[\kappa]$ by
Lemma~\ref{lem:coherence}. So $\FG$ is a homomorphism.
\end{proof}

\begin{proof}[Proof of \textup{4)}, bijectivity]
We claim $\FG$ is an involution: $\FG\circ\FG=\mathrm{id}$, which gives
bijectivity at once. By the bookkeeping, for each $\alpha$ there are
$p\in G$ and $\beta$ with $\alpha,\beta\in\supp(p)$ and
$p\force\nm Y_\alpha=\nm X_\beta$. It is
enough to show that then $p\force\nm Y_\beta\eqs\nm X_\alpha$.

First consider the region above $\Theta:=\hgt(p)$. Let $q\in G$ satisfy
$q\le p$ witnessed by the single block $B=[\Theta, \hgt(q))$, with top
level $f= f^q_{\hgt(q)}$. Using (P2)(c) first at $\beta$ and then at $\alpha$, together
with $X_\beta=Y_\alpha$ and $f\circ f=\mathrm{id}$ on $B$ by (P2)(b), we
obtain in $V[G]$
\[
Y_\beta\cap B=f[X_\beta\cap B]=f[Y_\alpha\cap B]
=f\bigl[f[X_\alpha\cap B]\bigr]=X_\alpha\cap B .
\]
As such conditions $q\in G$ have cofinal heights, this proves
$Y_\beta\setminus\Theta=X_\alpha\setminus\Theta$. On the other hand,
below $\Theta$ the internal clauses, $f^{p}_\Theta\circ f^{G}_\Theta=
\mathrm{id}$, and the fact that images of $\eqs$-equal sets under an
injection are $\eqs$, give
\[
Y_\beta\cap\Theta
\ \eqs\ f^{p}_{\Theta}\bigl[X_\beta\cap\Theta\bigr]
= f^{p}_{\Theta}\bigl[Y_\alpha\cap\Theta\bigr]
\ \eqs\ f^{p}_{\Theta}\bigl[f^{p}_{\Theta}[X_\alpha\cap\Theta]\bigr]
= X_\alpha\cap\Theta ,
\]
using $X_\beta=Y_\alpha$ for the first equality and the internal clause at
$\alpha$ for the second $\eqs$. Hence $\FG([\,Y_\alpha\,])=[X_\alpha]$.
\end{proof}

\begin{proof}[Proof of \textup{4)}, almost triviality]
Fix $\beta<\kappa$ and $H\subseteq\beta$ in $V[G]$, say $H=X_\gamma$. We claim $F_G([H]) = [f^G_\beta[H]]$. By
Lemma~\ref{lem:coherence}, for every $\delta\in[\beta,\kappa)$:
$Y_\gamma\cap\delta\eqs f^{G}_\delta[H]$, which latter equals $f^{G}_\beta[H]$
mod finite ($f^{G}_\delta\rest\beta\eqs f^{G}_\beta$, coherence  and $H\subseteq\beta$), so
$f^{G}_\delta[H]\eqs f^{G}_\beta[H]$. Thus $Y_\gamma\,\triangle\,f^{G}_\beta[H]$
has finite intersection with every $\delta<\kappa$, hence, $Y_\gamma\eqs f^{G}_\beta[H]$. Hence
$\FG\rest\Pw(\beta)/\Fin$ is induced by the involution $f^{G}_\beta$ (since $\cf(\kappa) > \omega$).
\end{proof}

\subsection{Nontriviality: item 5)}\label{sec:nontrivial}

Suppose some $p_0\in\mathbb P_{\kappa^{+}}$ forces
$\nm X\in[\kappa]^{\kappa}$ and $\nm E\colon\nm X\to\kappa$; it suffices, by
density, to find an extension of $p_0$ and a $\delta<\kappa$ such that the
extension forces $\nm E\rest(\nm X\cap\delta)\not\eqs
f^{G}_\delta\rest(\nm X\cap\delta)$.

\subsubsection{The fusion}
In what follows we are going to write $\Theta_n:=\hgt(p_n)$ and $\lambda_n:=|\supp(p_n)|$, and
$\mu_n:=2^{\lambda_n}$; note $\mu_n^{+}<\kappa$ since $\kappa$ is
inaccessible, and without loss of generality, all these cardinals are infinite. We build a decreasing sequence
$p_0\ge p_1\ge p_2\ge\cdots$ ($n<\omega$) that satisfies $\bigstar_{n,m}$ for all $n<m$, as follows.

Given $p_n$, first pick (by ${<}\kappa$-distributivity) an
extension $q\le p_n$ so that $q$  decides a
set $A_{n+1} \subseteq\nm X\cap[\Theta_n,\hgt(q))$ with
size $|A_{n+1}|=\mu_n^{+}$ as well as $\nm X \cap \Theta_n$, $\nm E \rest (\nm X \cap \Theta_n)$,  and for every $\sigma \in \supp(p_n)$ the set $\nm X_\sigma \cap \Theta_n$.

Now apply clause \ref{c2} from Claim \ref{claim:fusion} to $q\le p_n$ (with $\sigma = 0$) to obtain $r\le q$ with $r\in\mathbb
P_{\kappa^{+}}$ such that $r\le p_n$ is witnessed by the single pair
$\Theta_n<\hgt(r)$ and $f^{r}_{\hgt(r)}\rest\hgt(p_n)= f^{p_n}_{\hgt(p_n)}$. Starting with $r_0 := p_n$ and $r_1:= r$, repeat this construction for $\omega$-many steps and let $p_{n+1} = r_\omega$ be the $\bigstar$-limit of the resulting $r_i$'s. Then it is easily seen that $p_{n+1}$ decides $\nm X \cap \hgt(p_{n+1})$, $\nm E \rest (\nm X \cap \hgt(p_{n+1}))$, and $\nm X_\sigma \cap \hgt(p_{n+1})$ whenever $\sigma \in \supp(p_n)$.

The construction has the following properties:
\begin{enumerate}[label=(f\arabic*),start=0]
\item for every $n$, $p_n\in\mathbb P_{\kappa^+}$ and
$p_{n+1}\le p_n$ is witnessed by the single pair
$\Theta_n<\Theta_{n+1}$;
\item for every $n$, $p_{n+1}$ decides a set
$A_{n+1}\subseteq\nm X\cap[\Theta_n,\Theta_{n+1})$ with
$|A_{n+1}|=\mu_n^{+}$;
\item for every $n$, $p_{n+1}$ decides $\nm X_\sigma \cap \Theta_{n+1}$ for every $\sigma\in\supp(p_n)$,
\item for every $n$, $p_{n+1}$ decides the values $\nm X\cap\Theta_{n+1}$, 
$\nm E\rest(\nm X\cap\Theta_{n+1})$,
\item for all $n<m<\omega$, $(\bigstar_{n,m})$ holds.
\end{enumerate}

Let $\Theta_\omega:=\sup_n\Theta_n$, and let $p^{*}$ be the
$\bigstar$-limit of $\vec{p}=\langle p_n:n<\omega\rangle$ given by
Claim~\ref{claim:fusion}(C1). We will write $h$ for its top level, namely,
\[
h:=f^{p^{*}}_{\Theta_\omega}
=f^{p_0}_{\Theta_0}\cup\bigcup_{n<\omega}
f^{p_{n+1}}_{\Theta_{n+1}}\rest[\Theta_n,\Theta_{n+1});
\]
in particular, if $I_{n+1}:=[\Theta_n,\Theta_{n+1})$, then
$h\rest I_{n+1}=f^{p_{n+1}}_{\Theta_{n+1}}\rest I_{n+1}$ for every $n$, and
$p^{*}\le p_n$ is witnessed by the single pair
$\Theta_n<\Theta_\omega$. Since $p^*$ is a lower bound of $\vec{p}$, it
decides $\nm X\cap\Theta_\omega$ and
$\nm E\rest(\nm X\cap\Theta_\omega)$; denote their
values by $X_\omega$ and $g_0$, respectively.

We are going to define an alternative lower bound of the sequence $\langle p_n:n<\omega\rangle$ with height $\Theta_\omega$. As a lower bound, it will force $\nm E \rest (\nm X \cap \Theta_{\omega}) = g_0$. If its top level $h'$ satisfies $h' \rest X_\omega \neqs g_0$, then $g_0$ would not induce $\FG \rest \Pw(X_\omega)/\Fin$.

\subsubsection{The surgery}\label{subsub: surgery}
Fix $n$. Write $f:=f^{p_{n+1}}_{\Theta_{n+1}}\rest I_{n+1} =h \rest I_{n+1}$,
an involution of the interval $I_{n+1}$ by (P2)(b). To each
$a\in A_{n+1}$ associate its \emph{pattern}: the function
\[
\operatorname{pt}(a):\ \sigma\in\supp(p_n)\ \longmapsto\
\bigl(\,[a\in \nm X_\sigma],\ [f(a)\in \nm X_\sigma]\,\bigr)\in2\times2,
\]
where $[a\in\nm X_\sigma],[f(a)\in\nm X_\sigma]\in\{0,1\}$ denote the
truth values decided by $p_{n+1}$ as in (f2). There are at most $4^{\lambda_n}\le\mu_n$
patterns, so there is $C_{n+1}\subseteq A_{n+1}$ with $|C_{n+1}|=\mu_n^{+}$ such that
$\operatorname{pt}$ is constant on $C_{n+1}$. We now define an involution
$f'=f'_{n+1}$ from $I_{n+1}$ to $I_{n+1}$ and a
\emph{disagreement point} $i_{n+1}\in C_{n+1}$. Since $\mu_n^{+}$ is an infinite cardinal,
at least one of the following two cases occurs.

\emph{Case A: $|\{a\in C_{n+1}:f(a)=a\}|=\mu_n^{+}$.} Choose $i\ne i' \in C_{n+1}$ with
$f(i)=i$, $f(i')=i'$, and set $i_{n+1}:=i$. If
$g_0(i)\ne i$, set $f':=f$. If $g_0(i)=i$, set
\[
f'(x):=\begin{cases}
i'&x=i,\\
i&x=i',\\
f(x)&x\in I_{n+1}\setminus\{i,i'\}.
\end{cases}
\]

\emph{Case B: $|\{a\in C_{n+1}:f(a)\ne a\}|=\mu_n^{+}$.} Among the pairs
$\{a,f(a)\}$ ($a\in C_{n+1}$) there are $\mu_n^{+}$ many
pairwise disjoint ones. Choose two of
them, $\{i,f(i)\}$ and $\{i',f(i')\}$, with $i,i'\in C_{n+1}$ and the four
points $i,f(i),i',f(i')$ pairwise distinct, and set $i_{n+1}:=i$. If
$g_0(i)\ne f(i)$, set $f':=f$. If $g_0(i)=f(i)$, set
\[
f'(x):=\begin{cases}
f(i')&x=i,\\
i&x=f(i'),\\
f(i)&x=i',\\
i'&x=f(i),\\
f(x)&x\in I_{n+1}\setminus\{i,f(i),i',f(i')\}.
\end{cases}
\]

\begin{figure}[htbp]
\centering
\begin{tikzpicture}[x=.065\textwidth,y=1cm,
  every node/.style={font=\small},
  point/.style={circle,fill,inner sep=1.1pt}]
  \node at (2.6,1.65) {\textbf{Case A}};
  \node at (.9,1.1) {$f$};
  \node at (4.25,1.1) {$f'$};
  \node[point,label=below:{$i$}] (Ai) at (.35,0) {};
  \node[point,label=below:{$i'$}] (Aip) at (1.45,0) {};
  \draw (Ai) edge[loop above] (Ai);
  \draw (Aip) edge[loop above] (Aip);
  \node at (2.6,0) {$\longrightarrow$};
  \node[point,label=below:{$i$}] (Api) at (3.7,0) {};
  \node[point,label=below:{$i'$}] (Apip) at (4.8,0) {};
  \draw (Api)--(Apip);

  \node at (8.9,1.65) {\textbf{Case B}};
  \node at (7.25,1.1) {$f$};
  \node at (10.55,1.1) {$f'$};
  \node[point,label=above:{$i$}] (Bi) at (6.7,.1) {};
  \node[point,label=above:{$i'$}] (Bip) at (7.8,.1) {};
  \node[point,label=below:{$j$}] (Bj) at (6.7,-.85) {};
  \node[point,label=below:{$j'$}] (Bjp) at (7.8,-.85) {};
  \draw (Bi)--(Bj) (Bip)--(Bjp);
  \node at (8.9,-.375) {$\longrightarrow$};
  \node[point,label=above:{$i$}] (Bpi) at (10,.1) {};
  \node[point,label=above:{$i'$}] (Bpip) at (11.1,.1) {};
  \node[point,label=below:{$j$}] (Bpj) at (10,-.85) {};
  \node[point,label=below:{$j'$}] (Bpjp) at (11.1,-.85) {};
  \draw (Bpi)--(Bpjp) (Bpip)--(Bpj);
\end{tikzpicture}
\caption{The surgery when $g_0(i)=f(i)$; otherwise $f'=f$.
Loops denote fixed points and edges denote transpositions.
In Case B, $j=f(i)$ and $j'=f(i')$.
The map is unchanged outside the displayed points.}
\label{fig:surgery}
\end{figure}

In all cases, write $Q_{n+1} = \{j\in I_{n+1}: f(j)\neq f'(j)\}$. We note that 
$$f' \rest A_{n+1} \neq g_0 \rest A_{n+1}.$$
Also note that $Q_{n+1}$ is invariant under $f$ and $f'$ in both cases (as well as under their inverses).

\begin{lemma}\label{lem:surgery}
$f'$ is an involution
of $I_{n+1}$, and for every $\sigma\in\supp(p_n)$, $p_{n+1}$ forces:
\begin{equation}\label{eq:surgimage}
f'\bigl[\nm X_\sigma\cap I_{n+1}\bigr]=f\bigl[\nm X_\sigma\cap I_{n+1}\bigr].
\end{equation}
\end{lemma}

\begin{proof}
If $f'=f$ there is nothing to prove, so assume the modification was
performed. By definition, $f'$ is an involution. 

For \eqref{eq:surgimage}, work below $p_{n+1}$, and note that  it suffices to show (since the complement of $Q_{n+1}$ is invariant under the involutions $f$ and $f'$):
$f'[\dot X_\sigma \cap Q_{n+1}]=f[\dot X_\sigma \cap Q_{n+1}]$.

\emph{Case A.} Here $Q_{n+1}=\{i,i'\}$ with $f(i)=i$, $f(i')=i'$. Since
$\operatorname{pt}(i)=\operatorname{pt}(i')$ and for an $f$-fixed point
the two coordinates of the pattern coincide, we have
$[i\in\nm X_\sigma]=[i'\in\nm X_\sigma]$, i.e.\ either both $i,i'$ are in $\nm X_\sigma$, or neither is, and necessarily $f[\nm X_\sigma \cap \{i,i'\}] = f'[\nm X_\sigma \cap \{i,i'\}]$.

\emph{Case B.} Here $Q_{n+1}=\{i,f(i),i',f(i')\}$, four distinct points.
Recall that $i$ and $i'$ had the same pattern above, so
\[
[i\in \nm X_\sigma]=[i'\in \nm X_\sigma],\qquad
[f(i)\in \nm X_\sigma]=[f(i')\in \nm X_\sigma].
\]
By the definition of $f'$, we have that $ f[\{i,i'\}] =  f'[\{i,i'\}]$, and writing $j = f(i)$, $j' = f(i')$, $f[\{j,j'\}] = f'[\{j,j'\}]$. Since $p_{n+1}\Vdash`` \dot{X}_\sigma\cap \{i,i'\}=\emptyset$ or $\dot{X}_\sigma\cap \{i,i'\}=\{i,i'\}$" and $p_{n+1}\Vdash `` \dot{X}_\sigma\cap \{j,j'\}=\emptyset$ or $\dot{X}_\sigma\cap \{j,j'\}=\{j,j'\}$", we are done.

\end{proof}

To summarize, taking the union of the $f'_{n+1}$ for $n<\omega$, we proved the following.
\begin{claim}
	There exists an involution $h': \Theta_\omega \to \Theta_\omega$ such that 
	\begin{itemize}
		\item $\Theta_0$ and all intervals of the form $I_{n+1} = [\Theta_n, \Theta_{n+1})$ are invariant under $h'$,
	 
	\item for each $n$ $$ |\{ x \in [\Theta_n, \Theta_{n+1}): \ h(x) \neq h'(x)\}| \leq 4,$$
	\item and for each $n$ and $\sigma \in \supp(p_n)$
	$$p_{n+1} \Vdash h'[\nm X_\sigma \cap I_{n+1}] = h[\nm X_\sigma \cap I_{n+1}] (=f^{p_{n+1}}_{\Theta_{n+1}}[\nm X_\sigma \cap I_{n+1}]),$$
	\item and for each $n$, for the ground model objects $X_{n+1}$ and $g_0 \rest X_{n+1}$ (the values $p_{n+1}$ forces to $\nm X \cap [\Theta_n, \Theta_{n+1})$ and $\nm E \rest (\nm X \cap [\Theta_n, \Theta_{n+1}))$),
		$$h' \rest X_{n+1} \neq g_0 \rest  X_{n+1}.$$
	\end{itemize}
\end{claim}

\subsubsection{The lower bound}
We let $p'$ be the sequence of length $\kappa^+$ with
$\hgt(p')=\Theta_\omega$, with $(p')_0$ equal to
\[
\bigl\langle f^{p'}_{\delta}:\delta\le\Theta_\omega\bigr\rangle
:=\bigl\langle f^{p^{*}}_{\delta}:\delta<\Theta_\omega\bigr\rangle
{}^\frown\bigl\langle h'\bigr\rangle,
\]
$\supp(p'):=\supp(p^{*})$, and $(p')_\sigma:=(p^{*})_\sigma$ for
every $\sigma\in\supp(p^{*})$: $p'$ is $p^{*}$ with its top level
replaced by $h'$, all lower tower levels and all coordinate names
unchanged.

We claim that \emph{$p'\in \mathbb{P}_{\kappa^+}$.} The tower part: $h'$ is an involution
of $\Theta_\omega$. Moreover, $h'\rest [\Theta_n, \Theta_{n+1})$ and $h\rest [\Theta_n, \Theta_{n+1}) = f^{p_{n+1}}_{\Theta_{n+1}} \rest [\Theta_n, \Theta_{n+1})$ differ in at most four points, giving coherence.
 Let us show, by a simultaneous induction on
$\beta\le\kappa^{+}$, that
\begin{equation} \label{lastind}
p'\restriction \beta\in \mathbb{P}_\beta
\qquad\text{and}\qquad
p'\restriction\beta\le_{\mathbb P_\beta}p_n\restriction\beta\ \text{
for every }n<\omega,
\end{equation}
the latter witnessed by the single pair $\Theta_n<\Theta_\omega$. If
$\beta$ is a limit, both statements follow from the induction hypothesis
(the same witness works at every coordinate).

Successor case $\beta=\sigma+1$. If $\sigma\notin\supp(p')$, there is
nothing new to check. Suppose $\sigma\in\supp(p')$ is activated at stage $m$ (i.e.\
$\sigma\in\supp(p_n)$ for all $n\ge m$). Now any lower bound of the sequence
$\langle p_n\rest\sigma:n<\omega\rangle$, deciding
$\dot X_\sigma\cap\Theta_\omega$ in the same way, forces the following equality:
\begin{equation} \label{utseq}
h'\bigl[\nm X_\sigma \cap [\Theta_m, \Theta_\omega)\bigr] = h\bigl[\nm X_\sigma \cap [\Theta_m, \Theta_\omega)\bigr]
= \bigcup_{n>m}h\bigl[\nm X_\sigma\cap I_n\bigr] = \nm Y^{p^*}_\sigma \cap [\Theta_m, \Theta_\omega),
\end{equation}
so by our inductive hypothesis $p'\rest \sigma$ does, too. Moreover, $p_m \rest \sigma$ forces
$$\nm Y^{p^*}_\sigma \cap \Theta_m \eqs f^{p_m}_{\Theta_m}[ \nm X_\sigma \cap \Theta_m].$$ By coherence in $p'$'s tower, $h' \rest \Theta_m = f^{p'}_{\Theta_\omega} \rest \Theta_m \eqs f^{p_m}_{\Theta_m}$, and by the definition of $p'$ we have $\nm Y^{p^*}_\sigma = \nm Y^{p'}_\sigma$, so
$$\nm Y^{p'}_\sigma \cap \Theta_m \eqs h'[ \nm X_\sigma \cap \Theta_m] = f^{p'}_{\Theta_\omega}[ \nm X_\sigma \cap \Theta_m].$$
This, together with \eqref{utseq}, implies \eqref{lastind} for $\beta = \sigma+1$.

\emph{Conclusion.} To summarize, we proved the following.
\begin{claim} \label{llastcl}
	Whenever $p \Vdash (\dot X \in [\kappa]^\kappa) \wedge \dot E: \dot X \to \kappa$, there exists a condition $p' \leq p$ and a function $g_0$ such that $p'$ decides $\dot X \cap \hgt(p')$ and forces $\dot E \rest (\dot X \cap \hgt(p')) = g_0$, but for which  $f^{p'}_{\hgt(p')} \rest  (\dot X \cap \hgt(p')) \neqs g_0$.
\end{claim}

\begin{proof}[Proof of Theorem \ref{theorem: main}]
Choose the names $\langle\nm X_\gamma:1\le\gamma<\kappa^+\rangle$
during the recursion according to the bookkeeping supplied by
Corollary~\ref{cor:preservation}, so that, for every $\delta<\kappa^+$,
every nice $\mathbb P_\delta$-name for a subset of $\kappa$ occurs
cofinally often. Since $\mathbb P_{\kappa^+}$ is
$\kappa^+$-cc, every $\mathbb P_{\kappa^+}$-name for a subset of $\kappa$
is already a $\mathbb P_\delta$-name for some $\delta<\kappa^+$ and hence
appears later in this bookkeeping.

Cofinalities and cardinals are preserved by items 1) and 3) of
Lemma~\ref{lem:mainlemma} (Corollary~\ref{cor:preservation}), and
$\kappa$ stays inaccessible since no bounded subsets of $\kappa$ are
added. Item 4) gives the automorphism $\FG$ with
$\FG\circ\FG=\mathrm{id}$; it is trivial on $\Pw(\beta)/\Fin$ for every
$\beta<\kappa$, witnessed by the involution $f^{G}_\beta$, hence trivial
on $\Pw(X)/\Fin$ for every $X\in[\kappa]^{<\kappa}$.

It remains to show that $\FG$ is not trivial; we show the stronger
statement that for \emph{no} $X\in[\kappa]^{\kappa}$ and
$g\colon X\to\kappa$ does $g$ induce $\FG$ on $\Pw(X)/\Fin$.
Suppose, for the sake of contradiction, that in $V[G]$ some such $X$ and
$g$ did, i.e.\
$\FG([H])=[g[H]]$ for all $H\subseteq X$. Fix $\delta<\kappa$.
By almost triviality, $g[H]\eqs f^{G}_\delta[H]$ for every
$H\subseteq X\cap\delta$. If $g\rest(X\cap\delta)$ and
$f^{G}_\delta\rest(X\cap\delta)$ disagreed on infinitely many points,
Lemma~\ref{lemma: separate} would give an infinite $B\subseteq X\cap\delta$
with $g[B]\cap f^{G}_\delta[B]=\emptyset$. This contradicts
$g[B]\eqs f^{G}_\delta[B]$, since $f^{G}_\delta[B]$ is infinite.
Thus $g\rest(X\cap\delta)\eqs f^{G}_\delta\rest(X\cap\delta)$ for
\emph{every} $\delta<\kappa$. This contradicts Claim \ref{llastcl} applied to
$X$ and $E:=g$. Hence $\FG$ is a non-trivially non-trivial
automorphism of $\Pw(\kappa)/\Fin$. 

At this point we have everything at hand for the moreover part of Theorem \ref{theorem: main}. If $\lambda=\omega$, take $\Phi_\omega:=\FG$. Fix an infinite
$\lambda\le\kappa$ with $\lambda>\omega$. For $X\subseteq\kappa$, let
$Y\subseteq\kappa$ represent $\FG([X]_\Fin)$. By
Lemma~\ref{lem:coherence}, $Y$ \emph{threads}
\begin{equation} \label{threads}
\bigl\langle f^G_\alpha[X\cap\alpha]:\alpha<\kappa\bigr\rangle,
\qquad\text{that is,}\qquad
Y\cap\alpha\eqs f^G_\alpha[X\cap\alpha]
\quad(\alpha<\kappa).
\end{equation}
Define the induced map by
\[
\Phi_\lambda([X]_{I_\lambda}):=[Y]_{I_\lambda}.
\]
To see that this is well defined, suppose $|X\mathbin\triangle X'|<\lambda$
and let $Y,Y'$ thread the corresponding sequences. If
$|Y\mathbin\triangle Y'|\ge\lambda$, put
$\mu:=|X\mathbin\triangle X'|$. There is $\beta<\kappa$ such that
\[
|(Y\mathbin\triangle Y')\cap\beta|>\mu+\aleph_0.
\]
But $Y\cap\beta\eqs f^G_\beta[X\cap\beta]$ and
$Y'\cap\beta\eqs f^G_\beta[X'\cap\beta]$, while $f^G_\beta$ is a
bijection of $\beta$, so
\[
|(Y\mathbin\triangle Y')\cap\beta|
\le |(X\mathbin\triangle X')\cap\beta|+\aleph_0
\le\mu+\aleph_0,
\]
a contradiction. Thus $\FG$ preserves $I_\lambda$ and induces
$\Phi_\lambda$. Since $\FG$ is a Boolean-algebra involution, so is
$\Phi_\lambda$.

We finally show that $\Phi_\lambda\rest\Pw(X)/I_\lambda$ is nontrivial
for every $X\in[\kappa]^\kappa$. Suppose otherwise, as witnessed by
$g\colon X\to\kappa$. Shrinking $X$ to a $\kappa$-sized subset if necessary, we may take $g$ to be injective. By the
density arguments in
Section~\ref{sec:nontrivial}, we can find a discrete
$Y\in[\kappa\cap\operatorname{cof}(\omega)]^\kappa$ and
$\langle X_\alpha\subseteq X\cap\alpha:\alpha\in Y\rangle$ such that
	\begin{enumerate}
	\item for each $\alpha\in Y$, $|X_\alpha|=\aleph_0$;
		\item for distinct $\alpha,\beta\in Y$,
$X_\alpha\cap X_\beta=\emptyset$;
		\item \label{ha} for any $\alpha\in Y$ and $i\in X_\alpha$,
$g(i),f^G_\alpha(i)\in\alpha$ and $g(i)\neq f^G_\alpha(i)$;
		\item\label{it:growth} for any $\alpha<\beta$ in $Y$ and
$i\in X_\beta$,
$i,g(i),f^G_\beta(i)>\alpha$.
		\end{enumerate}
Using the injectivity of $g$, such a family is secured by running the
density argument of item 5) recursively in $V[G]$, choosing the members of
$Y$ and the corresponding sets $X_\alpha$ above all previously constructed
data.

By Lemma~\ref{lemma: separate}, for each $\alpha\in Y$ choose an infinite
$X'_\alpha\subseteq X_\alpha$ such that
\begin{equation} \label{disj}
	g[X'_\alpha]\cap f^G_\alpha[X'_\alpha]=\emptyset,
	\end{equation} and put
$A:=\bigcup_{\alpha\in Y}X'_\alpha$. Let
$C \subseteq \kappa$ thread
$\langle f^G_\alpha[A\cap\alpha]:\alpha<\kappa\rangle$ in the sense of \eqref{threads}, i.e. $\Phi_{\aleph_0}([A]_\Fin)=[C]_\Fin$ (in particular, $\Phi_\lambda([A]_{I_\lambda})=[C]_{I_\lambda}$ holds, too). Since $|g[A]\mathbin\triangle C|<\lambda\le\kappa$, due to the
assumption that $g$ induces
$\Phi_\lambda\rest\Pw(X)/I_\lambda$, and since $\kappa$ is regular, there is some $\xi<\kappa$ such that 
\begin{equation} \label{Cdef} g[A]\setminus \xi = C  \setminus \xi. \end{equation}

Let $\alpha < \beta$ be successive elements from $Y$ such that $\alpha > \xi$, so in particular (by \eqref{it:growth}),
\[
\min\bigl(X'_\beta\cup g[X'_\beta]\cup
f^G_\beta[X'_\beta]\bigr)> \alpha > \xi.
\]

Therefore, by \eqref{ha}  $g[X'_\beta] = g[A] \cap [\alpha, \beta)$,  so
\begin{equation}\label{tart}
g[X'_\beta]
=C  \cap [\alpha, \beta)
\end{equation}
by  \eqref{Cdef}.
On the other hand, since $C$ threads $\langle f^G_\alpha[A\cap\alpha]:\alpha<\kappa\rangle$, 
$$C \cap \beta \eqs f^G_\beta[A\cap\beta]$$

Now recall that 
$$A\cap\beta = \bigcup_{\gamma \leq \beta} X'_\gamma = \bigcup_{\gamma \leq \alpha} X'_\gamma \cup X_\beta' \subseteq \alpha \cup X_\beta'$$ by the definition of $A$, so 
$$ C \cap \beta \subseteq^* f_\beta^G[\alpha] \cup f_\beta^G[X'_\beta] \eqs \alpha \cup  f_\beta^G[X'_\beta].$$
So putting this together with $\eqref{tart}$,
$$g[X'_\beta] \subseteq^* f^G_\beta[X'_\beta],$$
which contradicts $\eqref{disj}$.
\end{proof}

To get a model of $\mathrm{MA}_{\aleph_1}$ and $\mathrm{OCA}_{\mathrm{T}}$ together with a nontrivial automorphism on $P(\kappa)/\fin$ for some cardinal $\kappa$, we can start with a model of $\mathrm{MA}_{\aleph_1}$ and $\mathrm{OCA}_{\mathrm{T}}$ and an inaccessible cardinal $\kappa$. First force with $\mathrm{Add}(\kappa^+,1)$ to arrange $2^\kappa=\kappa^+$; this forcing is $\kappa^+$-closed, so it preserves $\mathrm{MA}_{\aleph_1}$, $\mathrm{OCA}_{\mathrm{T}}$ and the inaccessibility of $\kappa$. Then we pass to a forcing extension as in Theorem \ref{theorem: main}. In the final model, $\mathrm{MA}_{\aleph_1}$ and $\mathrm{OCA}_{\mathrm{T}}$ are preserved since the forcing is $<\kappa$-distributive.

\section{Concluding remarks and open questions}\label{section: questions}

Let $\kappa$ be an inaccessible cardinal and let
$\Phi\colon\Pw(\kappa)/\Fin\to\Pw(\kappa)/\Fin$ be an automorphism.
Following Veli\v{c}kovi\'{c} \cite{Velickovic}, consider the ideal
\[
T(\Phi):=\{X\subseteq\kappa:
\Phi\rest\Pw(X)/\Fin\text{ is trivial}\}.
\]
Theorem~\ref{theorem: main} shows that it is consistent that
$T(\Phi)=[\kappa]^{<\kappa}$.

\begin{question}
Is it consistent that $T(\Phi)$ is a prime ideal on $\kappa$ that contains $[\kappa]^{<\kappa}$?
\end{question}

For $\kappa=\omega$, Veli\v{c}kovi\'{c} \cite{Velickovic} showed that it
is consistent for $T(\Phi)$ to be a nonprincipal prime ideal on $\omega$.

The following question is inspired by the fact that CH holds in all known models with a non-trivial automorphism of $\Pw(\omega_1)/\fin$.
\begin{question}
Is it consistent that all automorphisms of $\Pw(\omega)/\fin$ are trivial, yet there exists a non-trivial automorphism of $\Pw(\omega_1)/\fin$?
\end{question}

An easier test question is whether the failure of CH is consistent with the existence of a non-trivially non-trivial automorphism of $\Pw(\omega_1)/\fin$.

Construction principles developed by Brodsky and Rinot \cite{proxy} may be useful for the following question.

\begin{question}
In G\"odel's constructible universe $L$, is there a non-trivially non-trivial automorphism on $\Pw(\omega_1)/\Fin$, or on $\Pw(\kappa)/\Fin$ when $\kappa$ is the least inaccessible cardinal (if it exists)?
\end{question}

Recall that for an infinite cardinal $\kappa$, $\ell_2(\kappa)$ denotes the Hilbert space $\{(z_j)\in \mathbb{C}^\kappa: \sum_j |z_j|^2<\infty\}$. See \cite{CoronaRigidity} for an extensive survey. A variation of the model from Theorem \ref{theorem: main} may shed light on the following question.
\begin{question}
Is it consistent that there is an uncountable $\kappa$ such that the Calkin algebra $\mathcal{C}(\ell_2(\kappa))$ (the C*-algebra of bounded, linear operators on $\ell_2(\kappa)$ modulo the compact operators) admits an outer automorphism?
\end{question}

\section{Acknowledgment}
We want to thank Assaf Rinot for stimulating discussions on this subject, particularly during the Thematic Program on Set
Theoretic Methods in Algebra, Dynamics and Geometry at the Fields Institute in
the spring of 2023. We thank the organizers and the Fields Institute for their hospitality. We want to thank Juris Stepr\={a}ns for bringing the question to our attention. We thank Ilijas Farah and Stevo Todor\v{c}evi\'{c} for helpful comments and suggestions.

All the mathematical ideas are due to the authors. The assistance of AI (ChatGPT 5.6 and Fable 5) is exclusively in transcribing handwritten notes and proofreading. The second author thanks Binhui Zhang for the support.

\bibliographystyle{alpha}
\bibliography{bib}

%
%
%
%

\end{document}